\documentclass[reqno,a4paper,draft]{amsart}

\usepackage{enumitem}
\setenumerate{label=\textnormal{(\arabic*)}}

\usepackage{amsmath,amssymb,dsfont,verbatim,bm,array,mathtools}
\usepackage[latin1]{inputenc}
\usepackage{booktabs}
\usepackage[raggedright]{titlesec}
\usepackage{mathtools}

\titleformat{\chapter}[display]
{\normalfont\huge\bfseries}{\chaptertitlename\\thechapter}{20pt}{\Huge}
\titleformat{\section}
{\normalfont\Large\bfseries\center}{\thesection}{1em}{}
\titleformat{\subsection}
{\normalfont\large\bfseries}{\thesubsection}{1em}{}
\titleformat{\subsubsection}[runin]
{\normalfont\normalsize\bfseries}{\thesubsubsection}{1em}{}
\titleformat{\paragraph}[runin]
{\normalfont\normalsize\bfseries}{\theparagraph}{1em}{}
\titleformat{\subparagraph}[runin]
{\normalfont\normalsize\bfseries}{\thesubparagraph}{1em}{}
\titlespacing*{\chapter} {0pt}{50pt}{40pt}
\titlespacing*{\section} {0pt}{3.5ex plus 1ex minus .2ex}{2.3ex plus .2ex}
\titlespacing*{\subsection} {0pt}{3.25ex plus 1ex minus .2ex}{1.5ex plus .2ex}
\titlespacing*{\subsubsection}{0pt}{3.25ex plus 1ex minus .2ex}{1.5ex plus .2ex}
\titlespacing*{\paragraph} {0pt}{3.25ex plus 1ex minus .2ex}{1em}
\titlespacing*{\subparagraph} {\parindent}{3.25ex plus 1ex minus .2ex}{1em}

\input xypic
\xyoption{all}


\newtheorem{theorem}{Theorem}[section]
\newtheorem{lemma}[theorem]{Lemma}
\newtheorem{proposition}[theorem]{Proposition}

\newtheorem{conjecture}[theorem]{Conjecture}

\theoremstyle{definition}
\newtheorem{definition}[theorem]{Definition}

\theoremstyle{remark}
\newtheorem{remark}[theorem]{Remark}

\DeclareMathOperator{\Jac}{Jac}

\begin{document}
\title{Involutions and the Jacobian conjecture}
\author{Vered Moskowicz}
\address{Department of Mathematics, Bar-Ilan University, Ramat-Gan 52900, Israel.}
\email{vered.moskowicz@gmail.com}
\thanks{The author was partially supported by an Israel-US BSF grant \#2010/149}

\begin{abstract}
The famous Jacobian conjecture asks if an endomorphism $f$ of $K[x,y]$
($K$ is a characteristic zero field) 
that satisfies
$\Jac(f(x),f(y))\in K^*$ is invertible.

Let $\alpha$ be the exchange involution on $K[x,y]$: $\alpha(x)= y$ and $\alpha(y)= x$.
An $\alpha$-endomorphism $f$ of $K[x,y]$ is an endomorphism of $K[x,y]$ that preserves the involution $\alpha$:
$f \alpha= \alpha f$.
It was shown in ~\cite[Proposition 4.1]{mos val} that if $f$ is an $\alpha$-endomorphism 
that satisfies $\Jac(f(x),f(y))\in K^*$, 
then $f$ is invertible.
Based on this, we bring more results that imply that a given endomorphism $f$ 
that satisfies $\Jac(f(x),f(y))\in K^*$ 
and additional conditions involving involutions, 
is invertible.
\end{abstract}

\maketitle

\section{Introduction}
Let $K$ be a characteristic zero field.
The famous Jacobian conjecture asks if an endomorphism 
$f:K[x,y]\to K[x,y]$ that satisfies
$\Jac(f(x),f(y))\in K^*$ is invertible.

We suggest some partial answers to this conjecture,
based on the following previous result ~\cite[Proposition 4.1]{mos val}:
``If $f$ is an $\alpha$-endomorphism that satisfies
$\Jac(f(x),f(y))\in K^*$, then $f$ is invertible".

Where $\alpha$ is the exchange involution on $K[x,y]$,
$\alpha(x)= y$, $\alpha(y)= x$,
and an $\alpha$-endomorphism is an endomorphism of $K[x,y]$ 
that preserves the involution $\alpha$:
$f \alpha= \alpha f$.

Some of the results we bring here are mentioned, without a proof, 
in ~\cite{vered beta gamma}, while others are new. 

\section{Two equivalent conjectures to the Jacobian conjecture}
In Theorem \ref{TFAE} we show that the Jacobian conjecture is equivalent to the 
$\gamma,\delta$ conjecture \ref{conj gamma delta}
and to the $g,h$ conjecture \ref{conj g,h}.

By an involution on $K[x,y]$ we mean an automorphism of order $2$
(since $K[x,y]$ is commutative, any anti-automorphism is an automorphiam).
We wish to consider not only the exchange involution $\alpha$, 
but also any involution on $K[x,y]$. 

\begin{lemma}\label{conjugate}
Assume $\gamma$ is any involution on $K[x,y]$.
Then there exists an automorphism $g$ of $K[x,y]$ such that
$\gamma= g^{-1} \alpha g$.
\end{lemma}

In ~\cite{bell} J. Bell has sketched a proof for the analogue result in the first Weyl algebra. 
His sketch of proof is applicable to $K[x,y]$.

\begin{proof}
Use ~\cite[Proposition 8.9]{cohn} (or ~\cite{trees} with the fact that triangular automorphisms has infinite order), 
and a direct calculation that any linear automorphism of order $2$ is conjugate to $\alpha$.
\end{proof}

\begin{definition}
Let $f$ be an endomorphism of $K[x,y]$ and let 
$\gamma$ and $\delta$ be any two involutions on $K[x,y]$.
We say that $f$ is \begin{itemize}
\item a $\gamma$-endomorphism of $K[x,y]$, 
if $f \gamma= \gamma f$.
\item a $\gamma,\delta$-endomorphism of $K[x,y]$, 
if $f \gamma= \delta f$.
\end{itemize}
(If $\delta= \gamma$, then a 
$\gamma,\gamma$-endomorphism is just a $\gamma$-endomorphism).
\end{definition}

Notice that in the above definition, it is not assumed that the endomorphism $f$ satisfies
$\Jac(f(x),f(y))\in K^*$.

An immediate generalization of ~\cite[Proposition 4.1]{mos val} is as follows:

\begin{theorem}\label{starred true generalized}
Assume $\gamma$ and $\delta$ are two involutions on $K[x,y]$.
Assume $f$ is a $\gamma,\delta$-endomorphism of $K[x,y]$ 
that satisfies $\Jac(f(x),f(y))\in K^*$.
Then $f$ is invertible.
\end{theorem}

\begin{proof}
{}From Lemma \ref{conjugate}, there exists an automorphism $g$ of $K[x,y]$ 
such that $\gamma= g^{-1} \alpha g$,
and there exists an automorphism $h$ of $K[x,y]$ 
such that $\delta= h^{-1} \alpha h$.

$f$ is a $\gamma,\delta$-endomorphism, so by definition,
$f \gamma= \delta f$.
So, $f g^{-1} \alpha g= h^{-1} \alpha h f$.
Hence, $(h f g^{-1}) \alpha = \alpha (h f g^{-1})$.
Therefore, $h f g^{-1}$ is an $\alpha$-endomorphism.
{}From the chain rule we have 
$$
\Jac((hfg^{-1})(x), (hfg^{-1})(y)) \in K^*
$$

($g$ and $h$ are automorphisms of $K[x,y]$, 

so $\Jac(g^{-1}(x),g^{-1}(y)) \in K^*$ and $\Jac(h(x),h(y)) \in K^*$.

$\Jac(f(x),f(y)) \in K^*$ by assumption).

Apply ~\cite[Proposition 4.1]{mos val} to $hfg^{-1}$ 
and get that $hfg^{-1}$ is invertible, hence $f$ is invertible.
\end{proof}

\begin{remark}
Given two rings with involution $(R_1,\epsilon_1)$ and $(R_2,\epsilon_2)$,
we say that $f$ is an involutive endomorphism from 
$(R_1,\epsilon_1)$ to $(R_2,\epsilon_2)$,
if $f \epsilon_1= \epsilon_2 f$.

In particular, if $R_1= R_2$ then such $f$ is just an 
$\epsilon_1,\epsilon_2$-endomorphism.

Hence, Theorem \ref{starred true generalized} says the following:
Let $\gamma$ and $\delta$ be two involutions on $K[x,y]$.
Let $f$ be an involutive endomorphism from 
$(K[x,y],\gamma)$ to $(K[x,y],\delta)$
that satisfies $\Jac(f(x),f(y))\in K^*$.
Then $f$ is invertible.
\end{remark}

In view of Theorem \ref{starred true generalized}, given any endomorphism 
$f$ of $K[x,y]$ that satisfies $\Jac(f(x),f(y))\in K^*$,
one wishes to be able to find two involutions 
$\gamma$ and $\delta$ on $K[x,y]$, 
such that $f$ is a  $\gamma,\delta$-endomorphism.
Hence we suggest the following conjectures:

\begin{conjecture}[The $\gamma,\delta$ conjecture]\label{conj gamma delta}
Assume $f$ is an endomorphism of $K[x,y]$ that satisfies
$\Jac(f(x),f(y))\in K^*$.
Then there exist involutions $\gamma$ and $\delta$ on $K[x,y]$ such that 
$f$ is a $\gamma,\delta$-endomorphism.
\end{conjecture}  

\begin{conjecture}[The $g,h$ conjecture]\label{conj g,h}
Assume $f$ is an endomorphism of $K[x,y]$ that satisfies
$\Jac(f(x),f(y))\in K^*$.
Then there exist automorphisms $g$ and $h$ of $K[x,y]$ such that 
$hfg^{-1}$ is an $\alpha$-endomorphism.
\end{conjecture}

We have:
\begin{theorem}\label{TFAE}
TFAE: \begin{itemize}
\item [(1)] The Jacobian conjecture.
\item [(2)] The $\gamma,\delta$ conjecture.
\item [(3)] The $g,h$ conjecture.
\end{itemize}
Where $\gamma$ and $\delta$ are involutions on $K[x,y]$,
and $g$ and $h$ are automorphisms of $K[x,y]$.
\end{theorem}

\begin{proof}
Let $f$ be an endomorphism of $K[x,y]$ that satisfies
$\Jac(f(x),f(y))\in K^*$. 

$(1)\Longrightarrow(2)$: $f$ is invertible. 
Define $\gamma:= f^{-1} \alpha f$ and $\delta:= \alpha$, 
and get 
$f \gamma= f(f^{-1} \alpha f)= \alpha f= \delta f$.

$(2)\Longrightarrow(3)$: There exist involutions $\gamma$ and $\delta$ on $K[x,y]$ 
such that $f$ is a $\gamma,\delta$-endomorphism, 
namely $f \gamma= \delta f$.
{}From Lemma \ref{conjugate}, there exists an automorphism $g$ of $K[x,y]$ 
such that $\gamma= g^{-1} \alpha g$,
and there exists an automorphism $h$ of $K[x,y]$ such that
$\delta= h^{-1} \alpha h$.
 
Then $f \gamma= \delta f$ becomes
$(h f g^{-1}) \alpha = \alpha (h f g^{-1})$,
so $h f g^{-1}$ is an $\alpha$-endomorphism.

$(3)\Longrightarrow(1)$: There exist automorphisms $g$ and $h$ of $K[x,y]$ 
such that $hfg^{-1}$ is an $\alpha$-endomorphism.
Clearly, $\Jac((hfg^{-1})(x), (hfg^{-1})(y)) \in K^*$.

Apply ~\cite[Proposition 4.1]{mos val} to $hfg^{-1}$ and get that 
$hfg^{-1}$ is invertible, hence $f$ is invertible.
\end{proof}

The following Lemma is an analogue of ~\cite[Lemma 2.5]{vered beta gamma}; 
here we add the condition that $\Jac(f(x),f(y))\in K^*$, 
because we do not know how to prove that a 
$\gamma,\delta$-endomorphism is invertible without the condition on the Jacobian.

\begin{lemma}\label{lemma}
Assume $f$ is an endomorphism of $K[x,y]$ that satisfies
 
$\Jac(f(x),f(y))\in K^*$. 
Then: $f$ is a $\gamma,\delta$-endomorphism, 
where $\gamma$ and $\delta$ are involutions on $K[x,y]$ 
$\Longleftrightarrow$ $f$ is invertible.
\end{lemma}

The proof of this lemma actually appers in the above proof of Theorem \ref{TFAE}.

\begin{proof}
$\Longrightarrow$: There exist involutions $\gamma$ and $\delta$ on $K[x,y]$ 
such that $f \gamma= \delta f$.
 
$\gamma= g^{-1} \alpha g$ and $\delta= h^{-1} \alpha h$,
for some automorphisms $g$ and $h$ of $K[x,y]$.
 
Then $h f g^{-1}$ is an $\alpha$-endomorphism
that satisfies $\Jac((hfg^{-1})(x), (hfg^{-1})(y)) \in K^*$
(here we use the assumption $\Jac(f(x),f(y)) \in K^*$).

By ~\cite[Proposition 4.1]{mos val} $hfg^{-1}$ is invertible,
hence $f$ is invertible.

$\Longleftarrow$: Take $\gamma:= f^{-1} \alpha f$ and $\delta:= \alpha$, 
and get 
$f \gamma= f(f^{-1} \alpha f)= \alpha f= \delta f$.
\end{proof}

\section{Extension and restriction conditions}
In view of Theorem \ref{TFAE}, our (hopefully possible) mission is to prove that the 
$\gamma,\delta$ conjecture is true or to prove that the $g,h$ conjecture is true.

If each endomorphism $f$ of $K[x,y]$ that satisfies
$\Jac(f(x),f(y))\in K^*$ also satisfies the extension condition, 
then the $\gamma,\delta$ conjecture is true, see Theorem \ref{ext thm}.

If each endomorphism $f$ of $K[x,y]$ that satisfies
$\Jac(f(x),f(y))\in K^*$ also satisfies the restriction condition, 
then the $\gamma,\delta$ conjecture is true, see Theorem \ref{res thm}.

{}From now on we use the following notations:
$K$ will continue to denote a characteristic zero field, 
except in some results where we demand it to be the field of complex numbers.

Given an endomorphism $f$ of $K[x,y]$
(not necessarily satisfying $\Jac(f(x),f(y)) \in K^*$)
we denote $P:= f(x)$ and $Q:= f(y)$.
Denote by $T$ the image of $K[x,y]$ under $f$, namely $T= K[P,Q]$.
$T$ is a subalgebra of $K[x,y]$.
If $\Jac(P,Q) \in K^*$, then $T$ is isomorphic to $K[x,y]$.
Assume $\Jac(P,Q) \in K^*$ and denote by $\sigma_0$ the involution on $T$ 
which exchanges $P$ and $Q$, 
namely $\sigma_0(P)= Q$, $\sigma_0(Q)= P$ 
(extended in the obvious way to all of $T$).

\begin{remark}
Let $f$ be an endomorphism of $K[x,y]$ that satisfies
$\Jac(P,Q)\in K^*$. 
We do not know if $\sigma_0$ can be extended to an endomorphism of $K[x,y]$.
\end{remark}

\begin{definition}[The extension condition]
Let $f$ be an endomorphism of $K[x,y]$ that satisfies
$\Jac(P,Q)\in K^*$. 
We say that $f$ satisfies the extension condition if the involution 
$\sigma_0$ on $T$ can be extended to an involution on $K[x,y]$.
\end{definition}

Notice that in the above definition we demand that $\sigma_0$ can be extended 
not just to an endomorphism of $K[x,y]$,
but to an involution on $K[x,y]$ (an automorphism of $K[x,y]$ of order $2$).

\begin{theorem}[The extension theorem]\label{ext thm}
Assume $f$ is an endomorphism of $K[x,y]$ that satisfies
$\Jac(P,Q)\in K^*$. 
Then: $f$ satisfies the extension condition $\Longleftrightarrow$
$f$ is a $\gamma,\delta$-endomorphism of $K[x,y]$, 
where $\gamma$ and $\delta$ are involutions on $K[x,y]$.
\end{theorem}

\begin{proof}
$\Longrightarrow$: $f$ satisfies the extension condition, 
so the involution $\sigma_0$ on $T$ can be extended to an involution on $K[x,y]$, 
call it $\sigma$.
Therefore, we get that $f$ is an $\alpha,\sigma$-endomorphism of $K[x,y]$, 
because:
$$
(f \alpha)(x)= f(y)= Q= \sigma_0(P)= \sigma(P)= \sigma(f(x))= (\sigma f)(x) 
$$
and
$$
(f \alpha)(y)= f(x)= P= \sigma_0(Q)= \sigma(Q)= \sigma(f(y))= (\sigma f)(y).
$$
$\Longleftarrow$: $f$ is a $\gamma,\delta$-endomorphism of $K[x,y]$, 
where $\gamma$ and $\delta$ are involutions on $K[x,y]$.

Lemma \ref{lemma} implies that $f$ is invertible,
so $T= K[x,y]$.

By definition, $\sigma_0$ is the involution on $T$ given by
$\sigma_0(P)= Q$, $\sigma_0(Q)= P$,
hence $\sigma_0$ is an involution on $K[x,y]$,
so $f$ satisfies the extension condition (the extension of $\sigma_0$ to $K[x,y]$ is 
$\sigma_0$ itself).

(Remark: $T= K[x,y]$ so 
$x=\sum a_{ij}P^iQ^j$ and $y=\sum b_{ij}P^iQ^j$.
Therefore, 
$$
\sigma_0(x)=\sigma_0(\sum a_{ij}P^iQ^j)= \sum a_{ij} (\sigma_0(P))^i (\sigma_0(Q))^j=
\sum a_{ij} Q^i P^j.
$$
And similarly, $\sigma_0(y)= \sum b_{ij} Q^i P^j$).
\end{proof}

{}From Lemma \ref{lemma} and the above proof it is clear that
if $f$ is an endomorphism of $K[x,y]$ that satisfies
$\Jac(P,Q)\in K^*$, then: 
$f$ satisfies the extension condition $\Longleftrightarrow$
$f$ is invertible.
We also suggest to consider the following restriction condition. 

\begin{definition}[The restriction condition]
Let $f$ be an endomorphism of $K[x,y]$ that satisfies
$\Jac(P,Q)\in K^*$. 
We say that $f$ satisfies the restriction condition if 
$\alpha(P) \in T$ and $\alpha(Q) \in T$.
Equivalently, we say that $f$ satisfies the restriction condition if the exchange involution 
$\alpha$ on $K[x,y]$ when restricted to $T$ is an involution on $T$. 
\end{definition}

\begin{remark}
Let $f$ be an endomorphism of $K[x,y]$ that satisfies
$\Jac(P,Q)\in K^*$. 
We do not know if necessarily $\alpha(P) \in T$ 
and $\alpha(Q) \in T$.
\end{remark}

\begin{theorem}[The restriction theorem]\label{res thm}
Assume $f$ is an endomorphism of $K[x,y]$ that satisfies
$\Jac(P,Q)\in K^*$. 
Then: $f$ satisfies the restriction condition 
$\Longleftrightarrow$ $f$ is invertible.
\end{theorem}

\begin{proof}
$\Longrightarrow$: $f$ satisfies the restriction condition, 
so $\alpha$ restricted to $T$ is an involution on $T$.

Denote the restriction of $\alpha$ to $T$ by $\alpha_0$.

Since $T$ is isomorphic to $K[x,y]$ it follows from Lemma \ref{conjugate} that every 
involution on $T$ is conjugate (by an automorphism of $T$) to one chosen involution on $T$
(equivalently, any two involutions are conjugate);

in particular, there exists an automorphism $g_0$ of $T$ such that
$\sigma_0= g_0^{-1} \alpha_0 g_0$.
  
Therefore, 
$$
(f \alpha)(x)= f(y)= Q= \sigma_0(P)= \sigma_0(f(x))= (g_0^{-1} \alpha_0 g_0)(f(x))= 
(g_0^{-1} \alpha_0 g_0 f)(x)
$$
and
$$
(f \alpha)(y)= f(x)= P= \sigma_0(Q)= \sigma_0(f(y))= (g_0^{-1} \alpha_0 g_0)(f(y))= 
(g_0^{-1} \alpha_0 g_0 f)(y).
$$
Therefore, $f \alpha= (g_0)^{-1} \alpha_0 g_0 f$.

Then, $g_0 f \alpha= \alpha_0 g_0 f$,
so $g_0 f \alpha= \alpha g_0 f$,
namely $g_0 f$ is an $\alpha$-endomorphism of $K[x,y]$.

Since the Jacobian of $g_0(P),g_0(Q)$ with respect to $P,Q$,
denote it by $a$, is a non-zero scalar ($g_0$ is an automorphism of $T$)
and the Jacobian of $f(x),f(y)$ with respect to $x,y$,
denote it by $b$, is a non-zero scalar (by assumption 
$\Jac(f(x),f(y)) \in K^*$),
we get that
$\Jac((g_0 f)(x), (g_0 f)(y)) = ab \in K^*$.
($\Jac((g_0 f)(x), (g_0 f)(y))$ is the Jacobian of 
$(g_0 f)(x), (g_0 f)(y)$ with respect to $x,y$).

By ~\cite[Proposition 4.1]{mos val} $g_0 f$ is an automorphism of $K[x,y]$,
so $K[g_0(P),g_0(Q)]= K[x,y]$.

Hence,
$x= \sum a_{ij} (g_0(P))^i (g_0(Q))^j= g_0 (\sum a_{ij} P^iQ^j)$
and

$y= \sum b_{ij} (g_0(P))^i (g_0(Q))^j= g_0 (\sum b_{ij} P^iQ^j)$,
where $a_{ij}, b_{ij} \in K$.

This shows that $x,y \in T$, so $T =K[x,y]$, and we are done.

$\Longleftarrow$: $f$ is invertible,
so $T= K[x,y]$.
Hence $f$ satisfies the restriction condition, since trivially
$\alpha(P) \in K[x,y]= T$ and $\alpha(Q) \in K[x,y]= T$.
\end{proof}

One can generalize both the extension condition and the restriction condition to the following conditions 
and have results similar to Theorem \ref{ext thm} and Theorem \ref{res thm}.

More precisely:
Let $\epsilon_0$ be an involution on $T$.

\begin{definition}[The $\epsilon_0$ extension condition]
Let $f$ be an endomorphism of $K[x,y]$ that satisfies
$\Jac(P,Q)\in K^*$. 
We say that $f$ satisfies the $\epsilon_0$ extension condition if the involution 
$\epsilon_0$ on $T$ can be extended to an involution on $K[x,y]$.
\end{definition}

According to this definition, our previous extension condition is just the $\sigma_0$ extension condition.

\begin{theorem}[The $\epsilon_0$ extension theorem]\label{generalized ext thm}
Assume $f$ is an endomorphism of $K[x,y]$ that satisfies
$\Jac(P,Q)\in K^*$. 
Then: $f$ satisfies the $\epsilon_0$ extension condition 
$\Longleftrightarrow$ $f$ is a $\gamma,\delta$-endomorphism of $K[x,y]$, 
where $\gamma$ and $\delta$ are involutions on $K[x,y]$
$\Longleftrightarrow$ $f$ is invertible.
\end{theorem}

Let $\epsilon$ be an involution on $K[x,y]$.

\begin{definition}[The $\epsilon$ restriction condition]
Let $f$ be an endomorphism of $K[x,y]$ that satisfies
$\Jac(P,Q)\in K^*$. 
We say that $f$ satisfies the $\epsilon$ restriction condition 
if $\epsilon(P) \in T$ and $\epsilon(Q) \in T$.
Equivalently, we say that $f$ satisfies the $\epsilon$ restriction condition if the involution 
$\epsilon$ on $K[x,y]$, when restricted to $T$ is an involution on $T$. 
\end{definition}

According to this definition, our previous restriction condition is just the 
$\alpha$ restriction condition.

\begin{theorem}[The $\epsilon$ restriction theorem]\label{generalized res thm}
Assume $f$ is an endomorphism of $K[x,y]$ that satisfies
$\Jac(P,Q)\in K^*$. 
Then: $f$ satisfies the $\epsilon$ restriction condition 
$\Longleftrightarrow$ $f$ is invertible.
\end{theorem}

Summarizing, if $f$ is an endomorphism of $K[x,y]$ that satisfies
$\Jac(P,Q)\in K^*$, then TFAE: \begin{itemize}
\item $f$ satisfies the $\epsilon_0$ extension condition.
\item $f$ satisfies the $\epsilon$ restriction condition. 
\item $f$ is invertible.
\end{itemize}

\section{More results}
Recall the following result of Cheng-Mckay-Wang ~\cite[Theorem 1]{wang}:
``Let $K$ be the field of complex numbers. Assume $A,B \in K[x,y]$ satisfy
$\Jac(A,B) \in K^*$. 
If $R \in K[x,y]$ satisfies $\Jac(A,R)= 0$, then $R \in K[A]$".

Its analogue result in the first Weyl algebra over any characteristic zero field, 
not necessarily the field of complex numbers, can be found in ~\cite[Theorem 2.11]{ggv};
instead of the Jacobian take the commutator.

We shall use ~\cite[Theorem 1]{wang} in the proofs of Theorem \ref{generalized alpha is invertible}, 
Theorem \ref{generalized epsilon is invertible}, Theorem \ref{P symmetric} and 
Theorem \ref{generalized P symmetric}.

Therefore, in Theorem \ref{generalized alpha is invertible}, Theorem \ref{generalized epsilon is invertible}, 
Theorem \ref{P symmetric} and Theorem \ref{generalized P symmetric} we will demand that 
$K$ will be the field of complex numbers.

Actually, in the proofs of those theorems we use the $\epsilon$ restriction theorem 
\ref{generalized res thm}.

(In Theorem \ref{s k}, $K$ is any characteristic zero field.
In its proof we are not using the $\epsilon$ restriction theorem 
\ref{generalized res thm}).

\subsection{A generalized $\alpha$-endomorphism}
Of course, a $\gamma,\delta$-endomorphism is a generalization of an 
$\alpha$-endomorphism. 
Notice that in the definition of a $\gamma,\delta$-endomorphism $f$ 
(and in the definition of an $\alpha$-endomorphism $f$)
there is no special assumption on $\Jac(P,Q)$,
although in our results we add the assumption $\Jac(P,Q) \in K^*$. 

Now we suggest a generalization of an $\alpha$-endomorphism $f$ of $K[x,y]$
that satisfies $\Jac(P,Q) \in K^*$:

\begin{definition}[A generalized $\alpha$-endomorphism]
Let $f$ be an endomorphism of $K[x,y]$.
We say that $f$ is a generalized $\alpha$-endomorphism if the following two conditions are satisfied:
\begin{itemize}
\item [(1)] $\Jac(P,Q) \in K^*$.
\item [(2)] $\Jac(P,\alpha(P)) \in K^*$ or $\Jac(Q,\alpha(Q)) \in K^*$.
\end{itemize}
\end{definition}

Notice that a generalized $\alpha$-endomorphism $f$ is indeed a generalization of an 
$\alpha$-endomorphism $f$ of $K[x,y]$ 
that satisfies $\Jac(P,Q) \in K^*$:
Let $f$ be an $\alpha$-endomorphism of $K[x,y]$ that satisfies 
$\Jac(P,Q) \in K^*$.
Of course, 
$$
\alpha(P)= \alpha (f(x))= (\alpha f)(x)= (f \alpha)(x)= f(y)= Q.
$$
Hence, $\Jac(P,\alpha(P))= \Jac(P,Q) \in K^*$, as desired
(Actually also $\Jac(Q,\alpha(Q))= \Jac(Q,P) \in K^*$).

\begin{theorem}\label{generalized alpha is invertible}
Assume $K$ is the field of complex numbers.
If $f$ is a generalized $\alpha$-endomorphism, then $f$ is invertible.
\end{theorem}

\begin{proof}
$f$ is a generalized $\alpha$-endomorphism, hence by definition:
$\Jac(P,Q) \in K^*$ and 
$\Jac(P,\alpha(P)) \in K^*$ or $\Jac(Q,\alpha(Q)) \in K^*$. 

Assume w.l.o.g that $\Jac(P,\alpha(P)) \in K^*$.

Denote $\Jac(P,Q)= a$ and $\Jac(P,\alpha(P))= b$, 
where $a,b \in K^*$.
$$
0= 1-1= \Jac(P,\alpha(P)/b)- \Jac(P,Q/a)= \Jac(P, \alpha(P)/b - Q/a),
$$
so ~\cite[Theorem 1]{wang} implies that 
$\alpha(P)/b - Q/a = H(P)$, where $H(t) \in K[t]$.
Then $\alpha(P) = (bQ)/a + bH(P) \in T$.

{}From the chain rule, since 
$\Jac(P,Q)= a$ and $\Jac(\alpha(x),\alpha(y))= -1$,
we get $\Jac(\alpha(P),\alpha(Q)) = -a$.
$$
0= \Jac(\alpha(P),\alpha(Q)/(-a))- \Jac(\alpha(P),P/(-b))= 
\Jac(\alpha(P),\alpha(Q)/(-a)-P/(-b)),
$$
so ~\cite[Theorem 1]{wang} implies that 
$\alpha(Q)/(-a) - P/(-b) = G(\alpha(P))$, where $G(t) \in K[t]$.
Then,
$$
\alpha(Q) = (aP)/b - aG(\alpha(P))= (aP)/b - aG((bQ)/a + bH(P)) \in T.
$$
The restriction theorem \ref{res thm} implies that $f$ is invertible.
\end{proof}
An obvious generalization is as follows:

\begin{definition}[A generalized $\epsilon$-endomorphism]
Let $\epsilon$ be an involution on $K[x,y]$.
Let $f$ be an endomorphism of $K[x,y]$.

We say that $f$ is a generalized $\epsilon$-endomorphism 
if the following two conditions are satisfied:
\begin{itemize}
\item [(1)] $\Jac(P,Q) \in K^*$.
\item [(2)] $\Jac(P,\epsilon(P)) \in K^*$ or $\Jac(Q,\epsilon(Q)) \in K^*$.
\end{itemize}
\end{definition}

A generalized $\epsilon$-endomorphism is not a generalization of an $\epsilon$-endomorphism
that satisfies $\Jac(P,Q) \in K^*$; 
for example, 
$f(x)= x+y^2$ and $f(y)= y$
is a $\beta$-endomorphism (automorphism) that satisfies 
$\Jac(f(x),f(y))= 1$, where $\beta$ is the involution given by
$\beta(x)= x$ and $\beta(y)= -y$.
But, 
$$
\Jac(f(x), \beta(f(x)))= \Jac(x+y^2,x+y^2)= 0 
$$
and
$$
\Jac(f(y), \beta(f(y)))= \Jac(y,-y)= 0.
$$
Notice that $f$ is a generalized $\alpha$-endomorphism, since
$$
\Jac(f(y),\alpha(f(y)))= \Jac(y,x)= -1 \in K^*
$$.

\begin{theorem}\label{generalized epsilon is invertible}
Assume $K$ is the field of complex numbers.
If there exists an involution $\epsilon$ on $K[x,y]$ such that 
$f$ is a generalized $\epsilon$-endomorphism, 
then $f$ is invertible.
\end{theorem}

\begin{proof}
$f$ is a generalized $\epsilon$-endomorphism, hence by definition:
$\Jac(P,Q) \in K^*$ and
$\Jac(P,\epsilon(P)) \in K^*$ or $\Jac(Q,\epsilon(Q)) \in K^*$. 

Assume w.l.o.g that $\Jac(P,\epsilon(P)) \in K^*$.

Denote $\Jac(P,Q)= a$ and $\Jac(P,\epsilon(P))= b$, 
where $a,b \in K^*$.
$$
0= \Jac(P,\epsilon(P)/b)- \Jac(P,Q/a)= \Jac(P, \epsilon(P)/b - Q/a),
$$
so ~\cite[Theorem 1]{wang} implies that 
$\epsilon(P)/b - Q/a = H(P)$, where $H(t) \in K[t]$.
Then $\epsilon(P) = (bQ)/a + bH(P) \in T$.

{}From the chain rule, since
$\Jac(P,Q)= a \in K^*$ and
$\Jac(\epsilon(x),\epsilon(y)) \in K^*$,
we get 
$\Jac(\epsilon(P),\epsilon(Q)) \in K^*$.
$$
0= \Jac(\epsilon(P),\epsilon(Q)/(-a))- \Jac(\epsilon(P),P/(-b))= \Jac(\epsilon(P), \epsilon(Q)/(-a) - P/(-b)),
$$
so ~\cite[Theorem 1]{wang} implies that 
$\epsilon(Q)/(-a) - P/(-b) = G(\epsilon(P))$, where $G(t) \in K[t]$.
Then,
$$
\epsilon(Q) = (aP)/b - aG(\epsilon(P))= (aP)/b - aG((bQ)/a + bH(P)) \in T.
$$
The $\epsilon$ restriction theorem \ref{generalized res thm} implies that $f$ is invertible.
\end{proof}

The converse of Theorem \ref{generalized alpha is invertible} is not true; 
namely, an automorphism of $K[x,y]$ need not be a generalized $\alpha$-endomorphism. 
For example, 
$P:= h(x)=x+y$ and $Q:= h(y)=x-y$
is an automorphism of $K[x,y]$ 
which is not a generalized $\alpha$-endomorphism, 
since
$$
\Jac(P,\alpha(P))= 0 , \Jac(Q,\alpha(Q))= 0.
$$

$h$ is not a generalized $\alpha$-endomorphism, but it is a generalized 
$\beta$-endomorphism, 
where $\beta$ is the involution on $K[x,y]$ given by
$\beta(x)= x , \beta(y)= -y$.
Indeed,
$\Jac(P,\beta(P))= \Jac(x+y,x-y) \in K^*$

(and 
$\Jac(Q,\beta(Q))= \Jac(x-y,x+y) \in K^*$).

In view of this it seems natural to ask the following question:
Given an automorphism $g$ of $K[x,y]$, is there exists an involution 
$\epsilon$ on $K[x,y]$ such that 
$g$ is a generalized $\epsilon$-endomorphism?

Thus far we only managed to show that for every generator of the group of automorphisms 
of $K[x,y]$ the answer to this question is positive.
Indeed, \begin{itemize}
\item Let $g$ be linear with $g(x)=ax+by$, $a,b \in K$.
If $a=0$, then $g(x)=by$, so $g$ is a generalized $\alpha$-endomorphism:
$$
\Jac(g(x),\alpha(g(x)))= \Jac(by,bx) \in K^*.
$$
If $b=0$, then $g(x)=ax$, so $g$ is a generalized $\alpha$-endomorphism:
$$
\Jac(g(x),\alpha(g(x)))= \Jac(ax,ay) \in K^*.
$$
If both $a$ and $b$ are non-zero, then $g(x)=ax+by$,
so $g$ is a generalized $\beta$-endomorphism
($\beta(x)= x$ and $\beta(y)= -y$):
$$
\Jac(g(x),\beta(g(x)))= \Jac(ax+by,ax-by)= -2ab \in K^*.
$$
\item Let $g$ be triangular. Then it is clear that $g$ is a generalized 
$\alpha$-endomorphism.
\end{itemize}

\begin{remark}
A given endomorphism can be a generalized $\epsilon$-endomorphism,
for more that one involution $\epsilon$.
For example, let $g$ be the automorphism given by 
$g(x)=ax+by$ and $g(y)=cx+dy$,
where 
$ad-bc \neq 0$ and $a \neq \pm b$.
It is easy to see that 
$$
\Jac(g(x),\alpha(g(x)))= a^2-b^2 \in K^*.
$$
So $g$ is a generalized $\alpha$-endomorphism and we have already seen that 
$g$ is a generalized $\beta$-endomorphism
($\beta(x)= x$ and $\beta(y)= -y$).
\end{remark}

\subsection{$P$ is symmetric or skew-symmetric}
Denote the set of symmetric elements (with respect to $\alpha$) by $S_{\alpha}$ and 
denote the set of skew-symmetric elements (with respect to $\alpha$) by $K_{\alpha}$.
The set of symmetric elements is $K$-linearly spanned by 
$\{x^ny^m + x^my^n | n \geq m \}$, 
while the set of skew-symmetric elements is $K$-linearly spanned by 
$\{x^ny^m - x^my^n | n > m \}$.

\begin{lemma}\label{lemma a b}
\begin{itemize}
\item [(1)] If $a \in S_{\alpha}$ and $b \in S_{\alpha}$, 
then $\Jac(a,b) \in K_{\alpha}$.
\item [(2)] If $a \in K_{\alpha}$ and $b \in K_{\alpha}$, 
then $\Jac(a,b) \in K_{\alpha}$.
\item [(3)] If $a \in S_{\alpha}$ and $b \in K_{\alpha}$, 
then $\Jac(a,b) \in S_{\alpha}$.
\end{itemize}
\end{lemma}

The analogue result in the first Weyl algebra is also true (and is easier to prove), 
where instead of the Jacobian take the commutator.

\begin{proof}
We shall only prove $(1)$; the proofs of $(2)$ and $(3)$ are similar.
Write 
$$
a= \sum a_{ij}(x^iy^j + x^jy^i) , b= \sum b_{kl}(x^ky^l + x^ly^k).
$$
The Jacobian is $K$-linear, so 
$$
\Jac(a,b)= \sum \sum a_{ij}b_{kl} \Jac(x^iy^j + x^jy^i, x^ky^l + x^ly^k).
$$
Since the sum of skew-symmetric elements is skew-symmetric,
it suffices to show that each 
$\Jac(x^iy^j + x^jy^i, x^ky^l + x^ly^k)$ is skew-symmetric.

Indeed, a direct computation yields:

$\Jac(x^iy^j + x^jy^i, x^ky^l + x^ly^k)= $
$(li-kj)(x^{k+i-1}y^{j+l-1} - x^{j+l-1}y^{k+i-1}) + $

$(ik-jl)(x^{i+l-1}y^{j+k-1} - x^{j+k-1}y^{i+l-1})$.

\end{proof}

We again assume that $K$ is the field of complex numbers, because in our proof we wish to use 
~\cite[Theorem 1]{wang}.

\begin{theorem}\label{P symmetric}
Assume $K$ is the field of complex numbers.
Assume $f$ is an endomorphism of $K[x,y]$ that satisfies
$\Jac(P,Q)\in K^*$. 
Assume that one of the following conditions is satisfied:\begin{itemize}
\item $P$ is symmetric.
\item $P$ is skew-symmetric.
\item $Q$ is symmetric.
\item $Q$ is skew-symmetric.
\end{itemize}
Where by symmetric or skew-symmetric we mean symmetric or skew-symmetric with respect to $\alpha$.
Then $f$ is invertible.
\end{theorem}

\begin{remark}\label{remark P symmetric}
By Lemma \ref{lemma a b} the Jacobian of two symmetric or two skew-symmetric elements is skew-symmetric, 
hence it is impossible to have both $P$ and $Q$ symmetric or both 
$P$ and $Q$ skew-symmetric.

Actually, if $P$ is symmetric and $Q$ is skew-symmetric (or vice-versa), 
then it is immediate that such $f$ 
(= an endomorphism of $K[x,y]$ that satisfies $\Jac(P,Q)\in K^*$) is invertible:

Write $s:= P$ and $k:= Q$ ($s$ is symmetric and $k$ is skew-symmetric).
Define 
$g(x):= s+k$ and $g(y):= s-k$.
It is easy to see that $g$ is an $\alpha$-endomorphism 
of $K[x,y]$ that satisfies 

$\Jac(g(x),g(y)) \in K^*$: 
$g$ is an endomorphism of $K[x,y]$ that satisfies 

$\Jac(g(x),g(y)) \in K^*$:
Let 
$h(x)=x+y$ and $h(y)=x-y$. 
We have
$$
(fh)(x)= f(h(x))= f(x+y)= f(x)+f(y)= P+Q= s+k= g(x),
$$
and
$$
(fh)(y)= f(h(y))= f(x-y)= f(x)-f(y)= P-Q= s-k= g(y),
$$
so $g= fh$.
$$
\Jac(g(x),g(y))= \Jac((fh)(x),(fh)(y)) \in  K^*,
$$
since 
$\Jac(f(x),f(y)) \in K^*$ and 
$\Jac(h(x),h(y)) \in K^*$.
Another argument:
$$
\Jac(g(x),g(y))= \Jac(s+k,s-k)= 2\Jac(k,s) \in K^*.
$$                               
$g$ preserves $\alpha$:
$$
(g \alpha)(x)= g(\alpha(x))= g(y)= s-k= \alpha(s)+\alpha(k)= \alpha(s+k)= \alpha(g(x))= (\alpha g)(x)
$$
and
$$
(g \alpha)(y)= g(\alpha(y))= g(x)= s+k= \alpha(s)-\alpha(k)= \alpha(s-k)= \alpha(g(y))= (\alpha g)(y).
$$
{}From ~\cite[Proposition 4.1]{mos val} $g$ is invertible, hence $f$ is invertible
($f= gh^{-1}$, $g$ and $h^{-1}$ are automorphisms).
\end{remark}

\begin{proof}
Assume that $P$ is symmetric, namely $\alpha(P)= P \in T$.
Clearly, 
$$
\Jac(P,\alpha(Q))= \Jac(\alpha(P),\alpha(Q)) \in K^*.
$$
Write: $\Jac(P,Q)= a$ and $\Jac(P,\alpha(Q))= b$, 
where $a,b \in K^*$.
Then,
$$
\Jac(P,Q/a -\alpha(Q)/b)= \Jac(P,Q/a)- \Jac(P,\alpha(Q)/b)= 0,
$$ 
so from ~\cite[Theorem 1]{wang} we have 
$Q/a -\alpha(Q)/b= H(P)$ where $H(t) \in K[t]$.
Hence, $\alpha(Q)= bQ/a - bH(P) \in T$.
The restriction theorem \ref{res thm} implies that $f$ is invertible.

Showing that each of the other three conditions implies that $f$ is invertible is similar.
\end{proof}

Let $\epsilon$ be an involution on $K[x,y]$ and let $w \in K[x,y]$. 
$w$ is symmetric with respect to $\epsilon$ if $\epsilon(w)= w$, 
and $w$ is skew-symmetric with respect to $\epsilon$ if $\epsilon(w)= -w$.

One way to generalize Theorem \ref{P symmetric} is as follows:

\begin{theorem}\label{generalized P symmetric}
Assume $K$ is the field of complex numbers.
Assume $f$ is an endomorphism of $K[x,y]$ that satisfies
$\Jac(P,Q)\in K^*$. 
Assume that one of the following conditions is satisfied:\begin{itemize}
\item $P$ is symmetric.
\item $P$ is skew-symmetric.
\item $Q$ is symmetric.
\item $Q$ is skew-symmetric.
\end{itemize}
Where by symmetric or skew-symmetric we mean symmetric or skew-symmetric 
with respect to some involution $\epsilon$ on $K[x,y]$.
Then $f$ is invertible.
\end{theorem}

\begin{proof}
Assume that $P$ is symmetric with respect to $\epsilon$, 
namely $\epsilon(P)= P \in T$.
Clearly, 
$$
\Jac(P,\epsilon(Q))= \Jac(\epsilon(P),\epsilon(Q)) \in K^*.
$$
Write: $\Jac(P,Q)= a$ and $\Jac(P,\epsilon(Q))= b$, 
where $a,b \in K^*$.
Then,
$$
\Jac(P,Q/a -\epsilon(Q)/b)= \Jac(P,Q/a)- \Jac(P,\epsilon(Q)/b)= 0, 
$$
so from ~\cite[Theorem 1]{wang} we have 
$Q/a -\epsilon(Q)/b= H(P)$, where $H(t) \in K[t]$.
So $\epsilon(Q)= bQ/a - bH(P) \in T$.
The restriction theorem \ref{res thm} implies that $f$ is invertible.

Showing that each of the other three conditions implies that $f$ is invertible is similar.
\end{proof}

The converse of Theorem \ref{P symmetric} is not true; for example,
$g(x)=x$ and $g(y)=y+x^2$ is invertible, 
but non of 
$g(x),g(y)$ is symmetric or skew-symmetric with respect to $\alpha$.

It seems natural to ask the following question:
Given an automorphism $g$ of $K[x,y]$, is there exists an involution 
$\epsilon$ on $K[x,y]$ such that 
$g(x)$ or $g(y)$ is symmetric or skew-symmetric with respect to $\epsilon$?

Thus far we only managed to show that for every generator of the group of automorphisms 
of $K[x,y]$ the answer to this question is positive:
\begin{itemize}
\item Let $g$ be linear with $g(x)=ax+by$, 
$a,b \in K$.
If $a=0$, then $g(x)=by$ is symmetric with respect to the involution 
$x\mapsto -x$ and $y\mapsto y$.
If $b=0$, then $g(x)=ax$ is symmetric with respect to the involution 
$x\mapsto x$ and $y\mapsto -y$.
If both $a$ and $b$ are non-zero, then $g(x)=ax+by$ is symmetric 
with respect to the following involution $\epsilon$ given by
$\epsilon(x)= (b/a)y$ and $\epsilon(y)= (a/b)x$.
Indeed, 
$$
\epsilon(ax+by)= a\epsilon(x)+b\epsilon(y)= a(b/a)y+b(a/b)x= by+ax.
$$
(obviously, 
$$(\epsilon)^2(x)= \epsilon(\epsilon(x))= \epsilon((b/a)y)= 
(b/a) \epsilon(y)= (b/a)(a/b)x= x
$$
and
$$
(\epsilon)^2(y)= \epsilon(\epsilon(y))= \epsilon((a/b)x)= 
(a/b) \epsilon(x)= (a/b)(b/a)y= y.
$$
\item Let $g$ be triangular. Then it is clear that there exists an involution 
on $K[x,y]$ such that 
$g(x)$ or $g(y)$ is symmetric with respect to it
(for example, if $g(x)=x$ and $g(y)=y+x^3$, then take
$x\mapsto x$ and $y\mapsto -y$.
\end{itemize}

\begin{remark}
For a given endomorphism $g$ of $K[x,y]$, $g(x)$ can be symmetric (or skew-symmetric) 
with respect to more that one involution.
For example, let $g$ be the automorphism given by 
$g(x)=ax+by$ and $g(y)=cx+dy$,
where 
$ad-bc \neq 0$ and $a \neq \pm b$.

It is easy to see that 
$$
\Jac(g(x),\alpha(g(x)))= a^2-b^2 \in K^*.
$$
So $g$ is a generalized $\alpha$-endomorphism and we have already seen that 
$g$ is a generalized $\beta$-endomorphism
($\beta(x)= x$ and $\beta(y)= -y$).
\end{remark}

\begin{proposition}
Assume $f$ is an endomorphism of $K[x,y]$. 
The following conditions are equivalent:\begin{itemize}
\item [(1)] There exists an involution $\epsilon$ on $K[x,y]$ such that
$P$ is symmetric (skew-symmetric) with respect to $\epsilon$.
\item [(2)] There exists an automorphism $g$ of $K[x,y]$ such that
$g(P)$ is symmetric (skew-symmetric) with respect to $\alpha$.
\end{itemize}
\end{proposition}

\begin{proof}
$(1)\Rightarrow(2)$: 
By Lemma \ref{conjugate}, $\epsilon= g^{-1} \alpha g$ for some automorphism 
$g$ of $K[x,y]$.

$P$ is symmetric with respect to $\epsilon$: $\epsilon(P)=P$.

Therefore, $(g^{-1}\alpha g)(P)=P$, hence
$\alpha (g(P))=g(P)$, so $g(P)$ is symmetric with respect to $\alpha$.

$(2)\Rightarrow(1)$:
$g(P)$ is symmetric with respect to $\alpha$: 
$\alpha(g(P))=g(P)$.

Hence, $(g^{-1}\alpha g)(P)= P$.
Then $P$ is symmetric with respect to the involution $g^{-1} \alpha g$.

The skew-symmetric version can be proved similarly. 
\end{proof}

Another way to generalize Theorem \ref{P symmetric} is as follows;
notice that:\begin{itemize}
\item $K$ is not necessarily the field of complex numbers,
but any characteristic zero field.
\item There is no assumption on the Jacobian of
$f(x)$ and $f(y)$; however, we assume that there exist two special elements 
in the image of $f$ having a non-zero scalar Jacobian.
\end{itemize}

\begin{theorem}\label{s k}
Assume $f$ is an endomorphism of $K[x,y]$.
If there exist $s$ and $k$ in the image of $f$, $T$, such that
\begin{itemize}
\item $s$ is symmetric with respect to $\alpha$. 
\item $k$ is skew-symmetric with respect to $\alpha$.
\item $\Jac(s,k) \in K^*$,
\end{itemize}
then $f$ is invertible.
\end{theorem}

\begin{proof}
Define 
$g(x):= s+k$ and $g(y):= s-k$.
It is easy to see that $g$ is an $\alpha$-endomorphism 
of $K[x,y]$ that satisfies 
$$
\Jac(g(x),g(y))= \Jac(s+k,s-k)= 2\Jac(k,s) \in K^*.
$$
{}From ~\cite[Proposition 4.1]{mos val} $g$ is invertible.
Hence,
$$
T \supseteq K[s+k,s-k]= K[g(x),g(y)]= K[x,y],
$$
namely $f$ is surjective.

Recall that a surjective endomorphism of $K[x,y]$ is an automorphism
(see \cite[page 343]{cohn}), so $f$ is an automorphism.
\end{proof}

\section{Acknowledgements}
I wish to thank Prof. C. Valqui for working with me on the starred Dixmier conjecture 
\cite{mos val}, Prof. J. Bell for his note \cite{bell} concerning involutions on $A_1$, 
and Prof. L. Rowen and Prof. U. Vishne for being my advisors.

\bibliographystyle{plain}

\end{document}